\documentclass[12pt,reqno]{amsart}

\newcommand\version{\today}

\usepackage{amsmath,amsfonts,amsthm,amssymb,amsxtra,bbm}
\usepackage{fourier,setspace,graphicx,xcolor}
\usepackage[colorlinks=true]{hyperref}
\hypersetup{urlcolor=blue, linkcolor=blue, citecolor=red, anchorcolor=blue]}

\newtheorem{theorem}{Theorem}
\newtheorem{proposition}[theorem]{Proposition}
\newtheorem{lemma}[theorem]{Lemma}
\newtheorem{corollary}[theorem]{Corollary}

\theoremstyle{definition}

\theoremstyle{remark}

\newtheorem{remark}[theorem]{Remark}

\newcommand{\be}[1]{\begin{equation}\label{#1}}
\newcommand{\ee}{\end{equation}}
\renewcommand{\(}{\left(}
\renewcommand{\)}{\right)}
\newcommand{\1}{\mathbbm{1}}

\renewcommand{\epsilon}{\varepsilon}

\renewcommand{\phi}{\varphi}
\newcommand{\R}{\mathbb{R}}

\newcommand{\Sph}{\mathbb{S}}

\newcommand{\irN}[1]{\int_{\R^N}{#1}\,dx}
\newcommand{\Jlambda}[2]{\int_{\R^N}|#2|^\lambda\,#1(#2)\,d{#2}}

\newcommand{\pq}{p}

\begin{document}
\title[Reverse HLS --- \version]{Reverse Hardy-Littlewood-Sobolev inequalities}

\author[J.~Dolbeault]{Jean Dolbeault}
\address{\hspace*{-12pt}J.~Dolbeault: CEREMADE (CNRS UMR n$^\circ$ 7534), PSL research university, Universit\'e Paris-Dau\-phine, Place de Lattre de Tassigny, 75775 Paris 16, France}
\email{\href{mailto:dolbeaul@ceremade.dauphine.fr}{dolbeaul@ceremade.dauphine.fr}}

\author[R.~Frank]{Rupert L.~Frank}
\address{\hspace*{-12pt}R.~Frank: Mathematisches Institut, Ludwig-Maximilians Universit\"at M\"unchen, Theresienstr. 39, 80333 M\"unchen, Germany, and Department of Mathematics, California Institute of Technology, Pasadena, CA 91125, USA}
\email{\href{mailto:frank@math.lmu.de}{frank@math.lmu.de}}

\author[F.~Hoffmann]{Franca Hoffmann}
\address{\hspace*{-12pt}F.~Hoffmann: Department of Computing and Mathematical Sciences, California Institute of Technology, 1200 E California Blvd. MC 305-16, Pasadena, CA 91125, USA}
\email{\href{mailto:fkoh@caltech.edu}{fkoh@caltech.edu}}

\begin{abstract} This paper is devoted to a new family of reverse Hardy-Littlewood-Sobolev inequalities which involve a power law kernel with positive exponent. We investigate the range of the admissible parameters and characterize the optimal functions. A striking open question is the possibility of concentration which is analyzed and related with nonlinear diffusion equations involving mean field drifts.\end{abstract}

\keywords{Reverse Hardy-Littlewood-Sobolev inequalities; interpolation; non-linear diffusion; free energy; symmetrization; concentration; minimizer; existence of optimal functions; regularity; Euler--Lagrange equations.}
\subjclass[2010]{Primary: 35A23; Secondary: 26D15, 35K55, 46E35, 49J40.}
\maketitle\thispagestyle{empty}

We are concerned with the following minimization problem. For any $\lambda>0$ and any (measurable) function $\rho\ge0$ on $\R^N$, let
$$
I_\lambda[\rho] := \iint_{\R^N\times\R^N} \rho(x)\,|x-y|^\lambda\,\rho(y)\,dx\,dy \,.
$$
For $0<q<1$ we consider
$$
\mathcal C_{N,\lambda,q} := \inf\left\{ \frac{I_\lambda[\rho]}{\left( \int_{\R^N} \rho(x)\,dx \right)^\alpha\left( \int_{\R^N} \rho(x)^q\,dx \right)^{(2-\alpha)/q}} :\ 0\le\rho\in\mathrm L^1\cap\mathrm L^q(\R^N) \,,\ \rho\not\equiv 0 \right\},
$$
where
$$
\alpha:=\frac{2\,N-q\,(2\,N+\lambda)}{N\,(1-q)}\,.
$$
By convention, for any $p>0$ we use the notation $\rho\in\mathrm L^p(\R^N)$ if $\irN{|\rho(x)|^p}$ is finite. Note that $\alpha$ is determined by scaling and homogeneity: for given values of $\lambda$ and $q$, the value of $\alpha$ is the only one for which there is a chance that the infimum is positive. We are asking whether $\mathcal C_{N,\lambda,q}$ is equal to zero or positive and, in the latter case, whether there is a minimizer. We note that there are three regimes $q<2N/(2N+\lambda)$, $q=2N/(2N+\lambda)$ and $q>2N/(2N+\lambda)$, which respectively correspond to $\alpha>0$, $\alpha=0$ and \hbox{$\alpha<0$}. The case $q=2N/(2N+\lambda)$ has already been dealt with in~\cite{DZ15} by J.~Dou and M.~Zhu, and in~\cite{MR3666824} by Q.A.~Ng\^o and V.H.~Nguyen, who have explicitly computed $\mathcal C_{N,\lambda,q}$ and characterized all solutions of the corresponding Euler--Lagrange equation. In the following we will mostly concentrate on the other cases. Our main result is the following.
\begin{theorem}\label{main} Let $\lambda>0$, $q\in(0,1)$, $N\in \mathbb N^*$ and $\alpha$ as above. Then the inequality
\be{ineq:rHLS}
I_\lambda[\rho]\,\geq
\mathcal C_{N,\lambda,q}
\left( \int_{\R^N} \rho(x)\,dx \right)^\alpha\left( \int_{\R^N} \rho(x)^q\,dx \right)^{(2-\alpha)/q}
\ee
holds for any nonnegative function $\rho\in\mathrm L^1\cap\mathrm L^q(\R^N)$, for some positive constant $\mathcal C_{N,\lambda,q}$, if and only if $q>N/(N+\lambda)$. In this range, if either $N=1$, $2$ or $N\ge3$ and $q\ge \min\big\{1-2/N\,,\, 2N/(2N+\lambda)\big\}$, then the equality case is achieved by a radial nonnegative function $\rho\in\mathrm L^1\cap\mathrm L^q(\R^N)$.\end{theorem}
This theorem provides a necessary and sufficient condition for the validity of the inequality, namely $q>N/(N+\lambda)$ or equivalently $\alpha<1$. Concerning the existence of an optimizer, the theorem completely answers this question in dimensions $N=1$ and $N=2$. In dimensions $N\geq 3$ we obtain a sufficient condition for the existence of an optimizer, namely, $q\geq \min\big\{1-2/N,2N/(2N+\lambda)\big\}$. We emphasize that this is not a necessary condition and, in fact, in Proposition~\ref{Cor:Mstar} we prove existence in a slightly larger, but less explicit region. However, in the \emph{whole} region $q>N/(N+\lambda)$ we are able to prove the existence of an optimizer for a relaxed problem, with same optimal constant $\mathcal C_{N,\lambda,q}$, which allows for an additional Dirac mass at a single point. Therefore the question about existence of an optimizer in Theorem~\ref{main} is reduced to the much simpler, but still not obvious problem of whether the optimizer for this relaxed problem in fact has a Dirac mass. Fig.~\ref{Fig1} summarizes these considerations.

\medskip The Hardy-Littlewood-Sobolev (HLS) inequality is named after G.~Hardy and J.E.~Littlewood, see~\cite{MR1544927,MR1574995}, and S.L.~Sobolev, see~\cite{zbMATH03035784,sobolev1963theorem} (also see~\cite{MR0046395} for an early discussion of rearrangement methods applied to these inequalities). In 1983, E.H.~Lieb in~\cite{MR717827} proved the existence of optimal functions for negative values of $\lambda$ and established optimal constants. His proof requires an analysis of the invariances which has been systematized under the name of \emph{competing symmetries}, see~\cite{MR1038450}. A comprehensive introduction can be found in~\cite{MR1817225,burchard2009short}. Notice that rearrangement free proofs, which in some cases rely on the duality between Sobolev and HLS inequalities, have also been established more recently in various cases: see for instance~\cite{MR2659680,MR2925386,2014arXiv1404.1028J}. Standard HLS inequalities, which correspond to negative values of $\lambda$ in $I_\lambda[\rho]$, have many consequences in the theory of functional inequalities, particularly when the point comes to the identification of the optimal constants.

Relatively few results are known in the case $\lambda>0$. The conformally invariant case, \emph{i.e.}, $q=2N/(2N+\lambda)$, appears in~\cite{DZ15} and is motivated by some earlier results on the sphere (see references therein). Further results have been obtained in~\cite{MR3666824}, which still correspond to the conformally invariant case. Another range of exponents, which has no intersection with the one considered in the present paper, was studied earlier in~\cite[Theorem~G]{MR0121579}. Here we focus on a non conformally invariant family of interpolation inequalities corresponding to a given $\mathrm L^1(\R^N)$ norm. In a sense, these inequalities play for HLS inequalities a role analogous to Gagliardo-Nirenberg inequalities compared to Sobolev's conformally invariant inequality. 

\medskip Our study of~\eqref{ineq:rHLS} is motivated by the study of the nonnegative solutions of the evolution equation
\be{FP}
\frac{\partial\rho}{\partial t}=\Delta\rho^q+\,\nabla\cdot\(\rho\,\nabla W_\lambda\ast\rho\)
\ee
where the kernel is $W_\lambda(x):=\tfrac1\lambda\,|x|^\lambda$. Optimal functions for~\eqref{ineq:rHLS} provide energy minimizers for the \emph{free energy} functional
$$
\mathcal F[\rho]:=\frac12\int_{\R^N}\rho\,(W_\lambda\ast\rho)\,dx-\frac{1}{1-q}\int_{\R^N}\rho^q\,dx
=\frac{1}{2\lambda}\,I_\lambda[\rho]-\frac{1}{1-q}\int_{\R^N}\rho^q\,dx
$$ 
under a \emph{mass} constraint $M=\int_{\R^N}\rho\,dx$. It is indeed an elementary computation to check that a smooth solution $\rho(t,\cdot)$ of~\eqref{FP} with sufficient decay properties as $|x|\to+\infty$ is such that $M=\int_{\R^N}\rho(t,x)\,dx$ does not depend on $t$ while the free energy decays according to
$$
\frac d{dt}\mathcal F[\rho(t,\cdot)]:=-\int_{\R^N}\rho\left|\tfrac q{1-q}\nabla\rho^{q-1}-\nabla W_\lambda\ast\rho\right|^2\,dx\,.
$$
This identity allows us to identify the smooth stationary solutions as the solutions of
$$
\rho_s(x)=\(\mu+(W_\lambda\ast\rho_s)(x)\)^{-\frac1{1-q}}\,\quad\forall\,x\in\R^N\,,
$$
where $\mu$ is a constant which has to be determined by the mass constraint and observe that smooth minimizers of $\mathcal F$, whenever they exist, are stationary solutions.
\begin{corollary}\label{Cor:FreeEnergy} With the notations of Theorem~\ref{main}, $\mathcal F[\rho]$ is bounded from below on the set of nonnegative functions $\rho\in\mathrm L^1\cap\mathrm L^q(\R^N)$ if and only if $q>N/(N+\lambda)$, and in the range of parameters for which equality is achieved in~\eqref{ineq:rHLS}, if $\rho_*$ is optimal for~\eqref{ineq:rHLS}, then $\mathcal F[\rho]\ge\mathcal F[\rho_*]$ up to a scaling of $\rho_*$, if $\rho$ and $\rho_*$ have same mass.\end{corollary}
It is obvious that Eq.~\eqref{FP} is \emph{translation} invariant and this is also a natural property of~$\mathcal F[\rho]$. Concerning \emph{scaling} properties, we introduce $\rho_\tau(x):=\tau^N\rho(\tau\,x)$ for an arbitrary $\tau>0$ and observe that
$$
\tau\mapsto\mathcal F[\rho_\tau]=\frac{\tau^{-\lambda}}{2\lambda}\,I_\lambda[\rho]-\frac{\tau^{-N\,(1-q)}}{1-q}\irN{\rho^q}
$$
reaches a minimum at $\tau=\tau_0$ such that $\tau_0^{\lambda-N\,(1-q)}=\frac{I_\lambda[\rho]}{2N\,\,\irN{\rho^q}}$. As a consequence, we have that
$$
\mathcal F[\rho]\ge\mathcal F[\rho_{\tau_0}]=-\,\kappa\,\frac{\(\irN{\rho^q}\)^\frac\lambda{\lambda-N\,(1-q)}}{I_\lambda[\rho]^\frac{N\,(1-q)}{\lambda-N\,(1-q)}}\ge-\,\kappa\(\mathcal C_{N,\lambda,q}\, M^\alpha \)^{-\frac{N\,(1-q)}{\lambda-N\,(1-q)}}
$$
where $\kappa:=\(\tfrac1{1-q}-\tfrac{N}{\lambda}\)\(2N\)^\frac{N(1-q)}{\lambda-N\,(1-q)}>0$, which completes the proof of Corollary~\ref{Cor:FreeEnergy}.

Eq.~\eqref{FP} deserves some comments. It is a \emph{mean field} equation in the overdamped regime, in which the \emph{drift term} is an average of a spring force $\nabla W_\lambda(x)$ for any $\lambda>0$. The case $\lambda=2$ corresponds to linear springs obeying Hooke's law, while large $\lambda$ reflect a force which is small at small distances, but becomes very large for large values of $|x|$. In this sense, it is a \emph{strongly confining} force term. By expanding the diffusion term as $\Delta\rho^q=q\,\rho^{q-1}\(\Delta\rho+(q-1)\,\rho^{-1}\,|\nabla\rho|^2\)$ and considering $\rho^{q-1}$ as a diffusion coefficient, it is obvious that this \emph{fast diffusion} coefficient is large for small values of $\rho$ and has to be balanced by a very large drift term to avoid a \emph{runaway} phenomenon in which no stationary solutions may exist in $\mathrm L^1(\R^N)$. In the case of a drift term with linear growth as $|x|\to+\infty$, it is well known that the threshold is given by the exponent $q=1-2/N$ and it is also known according to, \emph{e.g.},~\cite{MR797051} for the pure fast diffusion case (no drift) that $q=1-2/N$ is the threshold for the global existence of nonnegative solutions in $\mathrm L^1(\R^N)$, with constant mass.

In the regime $q<1-2/N$, a new phenomenon appears which is not present in linear diffusions. As emphasized in~\cite{MR2282669}, the diffusion coefficient $\rho^{q-1}$ becomes small for large values of $\rho$ and does not prevent the appearance of singularities. This is one of the major issues investigated in this paper. Before giving details, let us observe that $W_\lambda$ is a convolution kernel which averages the solution and can be expected to give rise to a smooth effective potential term $V_\lambda=W_\lambda\ast\rho$ at $x=0$ if we consider a radial function~$\rho$. This is why we expect that $V_\lambda(x)=V_\lambda(0)+O\(|x|^2\)$ for $|x|$ small. With these considerations at hand, let us illustrate these ideas with a simpler model involving only a given, external potential~$V$. Assume that $u$ solves the \emph{fast diffusion with external drift} given by
$$
\frac{\partial u}{\partial t}=\Delta u^q+\,\nabla\cdot\big(u\,\nabla V\big)\,.
$$
To fix ideas, we shall take $V(x)=1+\frac12\,|x|^2+\frac1\lambda\,|x|^\lambda$, which is expected to capture the behavior of the potential $W_\lambda\ast\rho$ at $x=0$ and as $|x|\to+\infty$ when $\lambda\ge2$. Such an equation admits a free energy functional
$$
u\mapsto\irN{V\,u}-\frac1{1-q}\irN{u^q}\,,
$$
whose minimizers are, under the mass constraint $M=\irN u$, given by
$$
u_\mu(x)=\(\mu+V(x)\)^{-\frac1{1-q}}\,\quad\forall\,x\in\R^N\,.
$$
A linear spring would simply correspond to a fast diffusion Fokker--Planck equation when $V(x)=|x|^2/2$. One can for instance refer to~\cite{MR3497125} for a general account on this topic. In that case, it is straightforward to observe that the so-called \emph{Barenblatt profile} $u_\mu$ has finite mass if and only if $q>1-2/N$. For a general parameter $\lambda\ge2$, the corresponding integrability condition for $u_\mu$ is $q>1-\lambda/N$. But $q=1-2/N$ is also a threshold value for the regularity. Let us assume that $\lambda>2$ and $1-\lambda/N<q<1-2/N$, and consider the stationary solution~$u_\mu$, which depends on the parameter $\mu$. The mass of $u_\mu$ can be computed for any $\mu\ge-1$ by inverting the monotone decreasing function of $\mu\mapsto M(\mu)$ defined by
$$
M(\mu):=\irN{\(\mu+V(x)\)^{-\frac1{1-q}}}\le M(-1)=\irN{\(\tfrac12\,|x|^2+\tfrac1\lambda\,|x|^\lambda\)^{-\frac1{1-q}}}\,.
$$
Now, if one tries to minimize the free energy under the mass contraint $\irN u=M$ for an arbitrary $M>M(-1)$, it is left to the reader to check that the limit of a minimizing sequence is simply the measure $\big(M-M(-1)\big)\,\delta+u_{-1}$ where $\delta$ denotes the Dirac distribution centered at $x=0$. For the model described by Eq.~\eqref{FP}, the situation is by far more complicated because the mean field potential $V_\lambda=W_\lambda\ast\rho$ depends on the regular part $\rho$ and we have no simple estimate on a critical mass as in the case of an external potential~$V$.

Eq.~\eqref{FP} is a special case of a larger family of Keller-Segel type equations, which covers the cases with $q=1$ (linear diffusions) and $q\ge1$ (diffusions of porous medium type), and also the range of exponents $\lambda<0$. Of particular interest is the original parabolic--elliptic Keller--Segel system which corresponds in dimension $N=2$ to a limit case as $\lambda\to0$, in which the kernel is $W_0(x)=\frac1{2\pi}\,\log|x|$ and the diffusion exponent is $q=1$. The reader is invited to refer to~\cite{hoffmann2017keller} for a global overview of this class of problems and for a detailed list of references and applications. A research project related with the present paper,~\cite{CaDe2018}, focuses on uniqueness, gradient flows and the role of the free energy for the evolution equation~\eqref{FP} using a dyadic decomposition, and nicely complements the results presented here.

\section{Validity of the inequality}\label{Sec:Validity}

The following proposition gives a necessary and sufficient condition for inequality~\eqref{ineq:rHLS}.
\begin{proposition}\label{ineq}
Let $\lambda>0$.
\begin{enumerate}
\item If $0<q\le N/(N+\lambda)$, then $\mathcal C_{N,\lambda,q}=0$.
\item If $N/(N+\lambda)<q<1$, then $\mathcal C_{N,\lambda,q}>0$.
\end{enumerate}
\end{proposition}
The result in~(1) for $q<N/(N+\lambda)$ is obtained in~\cite{2018arXiv180301915C} using a different method. The result in~(1) for $q=N/(N+\lambda)$, as well as the result in~(2) is new.
\begin{proof}[Proof of Proposition~\ref{ineq}. Part (1).]
Let $\rho\ge0$ be bounded with compact support and let $\sigma\ge0$ be a smooth function with $\int_{\R^N}\sigma(x)\,dx =1$. With another parameter $M>0$ we consider
$$
\rho_\epsilon(x) = \rho(x) + M\,\epsilon^{-N}\,\sigma(x/\epsilon) \,.
$$
Then $\int_{\R^N} \rho_\epsilon(x)\,dx = \int_{\R^N} \rho(x)\,dx + M$ and, by simple estimates,
\begin{equation}
\label{eq:limitq}
\int_{\R^N} \rho_\epsilon(x)^q\,dx \to \int_{\R^N} \rho(x)^q\,dx
\quad\text{as}\quad\epsilon\to 0_+
\end{equation}
and
$$
I_\lambda[\rho_\epsilon] \to I_\lambda[\rho] + 2M\Jlambda\rho x
\quad\text{as}\quad\epsilon\to 0_+ \,.
$$
Thus, taking $\rho_\epsilon$ as a trial function,
\begin{equation}
\label{eq:functionalplusdelta}
\mathcal C_{N,\lambda,q}\le\frac{I_\lambda[\rho] + 2M\Jlambda\rho x}{\left( \int_{\R^N} \rho(x)\,dx + M \right)^\alpha\left( \int_{\R^N} \rho(x)^q\,dx \right)^{(2-\alpha)/q} }=:\mathcal Q[\rho,M]\,.
\end{equation}
This inequality is valid for any $M$ and therefore we can let $M\to+\infty$. If $\alpha>1$, which is the same as $q<N/(N+\lambda)$, we immediately obtain $\mathcal C_{N,\lambda,q}=0$ by letting $M\to+\infty$. If $\alpha=1$, \emph{i.e.}, $q=N/(N+\lambda)$, by taking the limit as $M\to+\infty$, we obtain
$$
\mathcal C_{N,\lambda,q}\le\frac{2\Jlambda\rho x}{\left( \int_{\R^N} \rho(x)^q\,dx \right)^{(2-\alpha)/q} } \,.
$$
Let us show that by a suitable choice of $\rho$ the right side can be made arbitrarily small. For any $R>1$, we take
$$
\rho_R(x):=|x|^{-(N+\lambda)}\,\mathbbm 1_{1\le|x|\le R}(x)\,.
$$
Then
$$
\int_{\R^N} |x|^\lambda\,\rho_R\,dx=\int_{\R^N}\rho_R^q\,dx=\left|\Sph^{N-1}\right|\,\log R
$$
and, as a consequence,
$$
\frac{\Jlambda{\rho_R}x}{\left( \int_{\R^N} \rho_R^{N/(N+\lambda)} \,dx \right)^{(N+\lambda)/N}}=\left(\left|\Sph^{N-1}\right|\,\log R\right)^{-\lambda/N} \to 0\quad\text{as}\quad R\to\infty \,.
$$
This proves that $\mathcal C_{N,\lambda,q}=0$ for $q=N/(N+\lambda)$.
\end{proof}

In order to prove that $\mathcal C_{N,\lambda,q}>0$ in the remaining cases we need the following simple bound.
\begin{lemma}\label{ineq2}
Let $\lambda>0$ and $N/(N+\lambda)<q<1$. Then there is a constant $c_{N,\lambda,q}>0$ such that for all $\rho\ge0$,
$$
\left( \int_{\R^N} \rho\,dx \right)^{1-\frac{N\,(1-q)}{\lambda\,q}}\(\Jlambda\rho x\)^\frac{N\,(1-q)}{\lambda\,q}\ge c_{N,\lambda,q}\(\int_{\R^N}\rho^q\,dx\)^{1/q}\,.
$$
\end{lemma}
\begin{proof}
Let $R>0$. Using H\"older's inequality, we obtain
$$
\int_{\{|x|<R\}} \rho^q\,dx\le\left( \int_{\R^N} \rho \,dx \right)^q |B_R|^{1-q} = C_1\left( \int_{\R^N} \rho \,dx \right)^q R^{N\,(1-q)}
$$
and
\begin{multline*}
\int_{\{|x|\ge R\}} \rho^q\,dx= \int_{\{|x|\ge R\}} \(|x|^\lambda\,\rho\)^q |x|^{-\lambda q} \,dx 
\le\left(\Jlambda\rho x\right)^q\left( \int_{\{|x|\ge R\}} |x|^{-\frac{\lambda\,q}{1-q}} \,dx \right)^{1-q} \\
= C_2\left(\Jlambda\rho x\right)^q R^{-\lambda q + N\,(1-q)} \,.
\end{multline*}
The fact that $C_2<\infty$ comes from the assumption $q>N/(N+\lambda)$, which is the same as $\lambda q/(1-q)>N$. To conclude, we add these two inequalities and optimize over $R$.
\end{proof}

\begin{proof}[Proof of Proposition~\ref{ineq}. Part (2).]
By rearrangement inequalities it suffices to prove the inequality for symmetric non-increasing $\rho$'s. For such functions, by the simplest rearrangement inequality,
\begin{equation*}
\int_{\R^N} |x-y|^\lambda \rho(y)\,dx\ge\Jlambda\rho x\quad\text{for all}\quad x\in\R^N \,.
\end{equation*}
Thus,
$$
I_\lambda[\rho]\ge\Jlambda\rho x\int_{\R^N}\rho\,dx\,.
$$
In the range $\frac N{N+\lambda}<q<\frac{2\,N}{2\,N+\lambda}$, we recall that by Lemma~\ref{ineq2}, we have for any symmetric non-increasing function $\rho$,
$$
\frac{I_\lambda[\rho]}{\(\irN{\rho(x)}\right)^\alpha}\ge\(\irN\rho\,dx\)^{1-\alpha}\Jlambda\rho x\ge c_{N,\lambda,q}^{2-\alpha}\(\int_{\R^N}\rho^q\,dx\)^\frac{2-\alpha}q
$$
because $2-\alpha=\frac{\lambda\,q}{N\,(1-q)}$. As a consequence, we obtain that
$$
\mathcal C_{N,\lambda,q}\ge c_{N,\lambda,q}^{2-\alpha}\,.
$$
\end{proof}
\begin{remark} The above computation explains a surprising feature of~\eqref{ineq:rHLS}: $I_\lambda[\rho]$ controls a product of two terms. However, in the range $N/(N+\lambda)<q<2N/(2N+\lambda)$ which corresponds to $\alpha\in(0,1)$, the problem is actually reduced to the interpolation of $\irN{\rho^q}$ between $\irN\rho$ and $\Jlambda\rho x$, which has a more classical structure.
\end{remark}
\begin{remark}\label{Rem:Cestimate} There is an alternative way to prove the inequality in the range $2N/(2N+\lambda)<q<1$ using the results from~\cite{DZ15,MR3666824}. We can indeed rely on H\"older's inequality to get that
$$
\(\irN{\rho(x)^q}\)^{1/q}\le\(\irN{\rho(x)^\frac{2\,N}{2\,N+\lambda}}\)^{\eta\,\frac{2\,N+\lambda}{2\,N}}\(\irN\rho\)^{1-\eta}
$$
with $\eta=\frac{2\,N\,(1-q)}{\lambda\,q}$. By applying the inequality
$$
I_\lambda[\rho]\ge\mathcal C_{N,\lambda,\frac{2\,N}{2\,N+\lambda}}\(\irN{\rho(x)^\frac{2\,N}{2\,N+\lambda}}\)^\frac{2\,N+\lambda}N
$$
shown in~\cite{DZ15,MR3666824} with an explicit constant, we obtain that
$$
\mathcal C_{N,\lambda,q}\ge\mathcal C_{N,\lambda,\frac{2\,N}{2\,N+\lambda}}=\pi^\frac\lambda2\,\frac{\Gamma\big(\frac N2-\frac\lambda2\big)}{\Gamma\big(N-\frac\lambda2\big)}\(\frac{\Gamma(N)}{\Gamma\big(\frac N2\big)}\)^{1-\frac\lambda N}\,.
$$
We notice that $\alpha=-\,2\,(1-\eta)/\eta$ is negative.
\end{remark}
\begin{remark}\label{Rem:c} In the proof of Lemma~\ref{ineq2}, the computation of the lower bound can be made more explicit and shows that
$$
c_{N,\lambda,q}\ge O\left(\big((N+\lambda)\,q-N\big)^{(1-q)/q}\right)\quad\mbox{as}\quad q\to N/(N+\lambda)_+\,.
$$
The optimal constant $c_{N,\lambda,q}$ can be explicitly computed after observing that a minimizer exists and that $x\mapsto(1+|x|^\lambda)^{-1/(1-q)}$ solves the Euler-Lagrange equation (after taking into account translations, scalings and homogeneity), thus realizing the equality case. An elementary computation shows that $\lim_{q\to N/(N+\lambda)_+}c_{N,\lambda,q}=0$. This limit is compatible with the fact that 
$$
\lim_{q\to N/(N+\lambda)_+}\mathcal C_{N,\lambda,q}=0
$$
because the map $(\lambda,q)\mapsto\mathcal C_{N,\lambda,q}$ is \emph{upper semi-continuous}. The proof of this last property goes as follows. Let us denote by
$$
\mathsf Q_{q,\lambda}[\rho]:=\frac{I_\lambda[\rho]}{\left( \int_{\R^N} \rho(x)\,dx \right)^\alpha\left( \int_{\R^N} \rho(x)^q\,dx \right)^{(2-\alpha)/q}}
$$
the energy quotient in which we emphasize the dependence in $q$ and $\lambda$. The infimum of $\mathsf Q_{q,\lambda}[\rho]$ over $\rho$ is $\mathcal C_{N,\lambda,q}$. Let $(q,\lambda)$ be a given point in $(0,1)\times(0,\infty)$ and let $(q_n,\lambda_n)$ be a sequence converging to $(q,\lambda)$. Let $\epsilon>0$ and choose a $\rho$ which is bounded, has compact support and is such that $\mathsf Q_{q,\lambda}[\rho]\leq \mathcal C_{N,\lambda,q} + \epsilon$. Then, by the definition as an infimum, $\mathcal C_{N,q_n,\lambda_n} \leq\mathsf Q_{q_n,\lambda_n}[\rho]$.
On the other hand, the assumptions on $\rho$ easily imply that $\lim_{n\to\infty}\mathsf Q_{q_n,\lambda_n}[\rho]=\mathsf Q_{q,\lambda}[\rho]$. We conclude that $\limsup_{n\to\infty}\mathcal C_{N,q_n,\lambda_n}\leq\mathcal C_{N,\lambda,q} + \epsilon$. Since $\epsilon$ is arbitrary, we obtain the claimed upper semi-continuity.\end{remark}

\section{Existence of minimizers}\label{Sec:Existence}

We now investigate whether there are minimizers for $\mathcal C_{N,\lambda,q}$ if $N/(N+\lambda)<q<1$. As mentioned before, the conformally invariant case $q=2N/(2N+\lambda)$ has been dealt with before and will be excluded from our considerations. We start with the simpler case $2N/(2N+\lambda)<q<1$, which corresponds to $\alpha<0$.
\begin{proposition}\label{opt0}
Let $\lambda>0$ and $2N/(2N+\lambda)<q<1$. Then there is a minimizer for $\mathcal C_{N,\lambda,q}$.
\end{proposition}
\begin{proof}
Let $(\rho_j)_{j\in\mathbb N}$ be a minimizing sequence. By rearrangement inequalities we may assume that the $\rho_j$ are symmetric non-increasing. By scaling and homogeneity, we may also assume that
$$
\int_{\R^N} \rho_j(x)\,dx = \int_{\R^N}\rho_j(x)^q\,dx = 1
\quad\text{for all}\,j\in\mathbb N\,.
$$
This together with the symmetric non-increasing character implies that 
$$
\rho_j(x)\le C\,\min\left\{ |x|^{-N},\,|x|^{-N/q}\right\}
$$
with $C$ independent of $j$. By Helly's selection theorem we may assume, after passing to a subsequence if necessary, that $\rho_j\to\rho$ almost everywhere. The function $\rho$ is symmetric non-increasing and satisfies the same upper bound as $\rho_j$.

By Fatou's lemma we have
$$
\liminf_{j\to\infty} I_\lambda[\rho_j]\ge I_\lambda[\rho] 
\quad\text{and}\quad
1\ge \int_{\R^N} \rho(x)\,dx \,.
$$
To complete the proof we need to show that $\int_{\R^N} \rho(x)^q\,dx =1$ (which implies, in particular, that $\rho\not\equiv 0$) and then $\rho$ will be an optimizer.

Modifying an idea from~\cite{MR3273640} we pick $\pq\in\(N/(N+\lambda),q\)$ and apply~\eqref{ineq:rHLS} with the same~$\lambda$ and $\alpha(p)=\big(2\,N-p\,(2\,N+\lambda)\big)\big/\big(N\,(1-p)\big)$ to get
$$
I_\lambda[\rho_j]\ge\mathcal C_{N,\lambda,\pq}\left( \int_{\R^N} \rho_j^{\pq}\,dx \right)^{(2-\alpha(p))/\pq} \,.
$$
Since the left side converges to a finite limit, namely $\mathcal C_{N,\lambda,q}$, we find that the $\rho_j$ are uniformly bounded in $\mathrm L^{\pq}(\R^N)$ and therefore we have as before
$$
\rho_j(x)\le C'\,|x|^{-N/\pq} \,.
$$
Since $\min\left\{|x|^{-N},|x|^{-N/\pq}\right\}\in\mathrm L^q(\R^N)$, we obtain by dominated convergence
$$
\int_{\R^N} \rho_j^q \,dx \to \int_{\R^N} \rho^q\,dx \,,
$$
which, in view of the normalization, implies that $\int_{\R^N} \rho(x)^q\,dx =1$, as claimed.
\end{proof}

Next we prove the existence of minimizers in the range $N/(N+\lambda)<q<2N/(2N+\lambda)$ by considering the \emph{minimization of a relaxed problem}. The idea behind the relaxed problem is to allow $\rho$ to contain a Dirac function at the origin. The motivation for this comes from the proof of the first part of Proposition~\ref{ineq}. The expression on the right side of~\eqref{eq:functionalplusdelta} comes precisely from $\rho$ together with a delta function of strength $M$ at the origin. We have seen that in the regime $q\le N/(N+\lambda)$ (that is, $\alpha\ge1$) it is advantageous to increase~$M$ to infinity. This is no longer so if $N/(N+\lambda)<q<2N/(2N+\lambda)$. While it is certainly disadvantageous to move $M$ to infinity, it depends on $\rho$ whether the optimum $M$ is $0$ or a positive finite value.

Let
$$
\mathcal C_{N,\lambda,q}^{\rm{rel}} := \inf\left\{\mathcal Q[\rho,M]\,:\,0\le\rho\in\mathrm L^1\cap\mathrm L^q(\R^N) \,,\ \rho\not\equiv 0 \,,\ M\ge0 \right\}
$$
where $\mathcal Q[\rho,M]$ is defined by~\eqref{eq:functionalplusdelta}. We know that $\mathcal C_{N,\lambda,q}^{\rm{rel}}\le\mathcal C_{N,\lambda,q}$ by restricting the minimization to $M=0$. On the other hand,~\eqref{eq:functionalplusdelta} gives $\mathcal C_{N,\lambda,q}^{\rm{rel}}\ge\mathcal C_{N,\lambda,q}$. Therefore,
$$
\mathcal C_{N,\lambda,q}^{\rm{rel}}= \mathcal C_{N,\lambda,q} \,,
$$
which justifies our interpretation of $\mathcal C_{N,\lambda,q}^{\rm{rel}}$ as a \emph{relaxed minimization} problem. Let us start with a preliminary observation.
\begin{lemma}\label{sym-pos} Let $\lambda>0$ and $N/(N+\lambda)<q<1$. If $\rho\ge0$ is an optimal function for either $\mathcal C_{N,\lambda,q}^{\rm{rel}}$ (for an $M>0$) or $\mathcal C_{N,\lambda,q}$ (with $M=0$), then $\rho$ is radial (up to a translation), monotone non-increasing and positive almost everywhere on $\R^d$.\end{lemma}
\begin{proof} Since $\mathcal C_{N,\lambda,q}$ is positive, we observe that $\rho$ is not identically $0$. By rearrangement inequalities and up to a translation, we know that~$\rho$ is radial and monotone non-increa\-sing. Assume by contradiction that $\rho$ vanishes on a set $E\subset\R^N$ of finite, positive measure. Then
$$
\mathcal Q\big[\rho,M+\epsilon\,\1_E\big]=\mathcal Q[\rho,M]\(1-\frac{2-\alpha}q\,\frac{|E|}{\irN{\rho(x)^q}}\,\epsilon^q+o(\epsilon^q)\)
$$
as $\varepsilon\to0_+$, a contradiction to the minimality for sufficiently small $\epsilon>0$.
\end{proof}
\begin{proposition}\label{opt1}
Let $\lambda>0$ and $N/(N+\lambda)<q<2N/(2N+\lambda)$. Then there is a minimizer for $\mathcal C_{N,\lambda,q}^{\rm{rel}}$.
\end{proposition}
We will later show that for $N=1$ and $N=2$ there is a minimizer for the original problem $\mathcal C_{N,\lambda,q}$ in the full range of $\lambda$'s and $q$'s covered by Proposition~\ref{opt1}. If $N\ge3$, the same is true under an additional restriction.

\begin{proof}[Proof of Proposition~\ref{opt1}]
The beginning of the proof is similar to that of Proposition~\ref{opt0}. Let $(\rho_j,M_j)$ be a minimizing sequence. By rearrangement inequalities we may assume that $\rho_j$ is symmetric non-increasing. Moreover, by scaling and homogeneity, we may assume that
$$
\int_{\R^N} \rho_j \,dx +M_j = \int_{\R^N} \rho_j^q = 1 \,.
$$
In a standard way this implies that
$$
\rho_j(x)\le C\,\min\left\{ |x|^{-N}, |x|^{-N/q}\right\}
$$
with $C$ independent of $j$. By Helly's selection theorem we may assume, after passing to a subsequence if necessary, that $\rho_j\to\rho$ almost everywhere. The function $\rho$ is symmetric non-increasing and satisfies the same upper bound as $\rho_j$. Passing to a further subsequence, we can also assume that $(M_j)$ and $\(\irN{\rho_j}\)$ converge and define $M:=L+\lim_{j\to\infty}M_j$ where $L=\lim_{j\to\infty}\irN{\rho_j}-\irN\rho$, so that $\irN\rho+M=1$.
In the same way as before, we show that
$$
\int_{\R^N} \rho(x)^q\,dx = 1\,.
$$

We now turn our attention to the $\mathrm L^1$-term. We cannot invoke Fatou's lemma because $\alpha\in(0,1)$. The problem with this term is that $|x|^{-N}$ is not integrable at the origin and we cannot get a better bound there. We have to argue via measures, so let $d\mu_j(x) := \rho_j(x)\,dx$. Because of the upper bound on $\rho_j$ we have
$$
\mu_j\(\R^N\setminus B_R(0)\) = \int_{\{|x|\ge R\}} \rho_j(x)\,dx\le C \int_{\{|x|\ge R\}} \frac{dx}{|x|^{N/q}} = C'\,R^{-N\,(1-q)/q} \,.
$$
This means that the measures are tight. After passing to a subsequence if necessary, we may assume that $\mu_j\to\mu$ weak * in the space of measures on $\R^N$. Tightness implies that $\mu(\R^N)=L+\irN\rho$. Moreover, since the bound $C\,|x|^{-N/q}$ is integrable away from any neighborhood of the origin, we see that $\mu$ is absolutely continuous on $\R^N\setminus\{0\}$ and $d\mu/dx =\rho$. In other words,
$$
d\mu = \rho\,dx +L\,\delta\,.
$$

Using weak convergence in the space of measures one can show that
$$
\liminf_{j\to\infty} I_\lambda[\rho_j]\ge I_\lambda[\rho] + 2M\Jlambda\rho x\,.
$$
Finally, by Fatou's lemma,
$$
\liminf_{j\to\infty}\Jlambda{\rho_j}x\ge\int_{\R^N} |x|^\lambda\left(\rho(x)\,dx +L\,\delta\right) =\Jlambda\rho x\,. 
$$
Thus, 
$$
\liminf_{j\to\infty}\mathcal Q[\rho_j,M_j]\ge\mathcal Q[\rho,M]\,.
$$
By definition of $\mathcal C_{N,\lambda,q}^{\rm{rel}}$ the right side is bounded from below by $\mathcal C_{N,\lambda,q}^{\rm{rel}}$. On the other hand, by choice of $\rho_j$ and $M_j$ the left side is equal to $\mathcal C_{N,\lambda,q}^{\rm{rel}}$. This proves that $(\rho, M)$ is a minimizer for $\mathcal C_{N,\lambda,q}^{\rm{rel}}$.
\end{proof}

Next, we show that under certain assumptions a minimizer $(\rho_*,M_*)$ for the relaxed problem must, in fact, have $M_*=0$ and is therefore a minimizer of the original problem.
\begin{proposition}\label{opt2}
Let $\lambda>0$ and $N/(N+\lambda)<q<2N/(2N+\lambda)$. If either $N=1$, $2$ or $N\ge3$ and
$\lambda >2N/(N-2)$, assume, in addition, that $q\ge1-2/N $. If $(\rho_*,M_*)$ is a minimizer for $\mathcal C_{N,\lambda,q}^{\rm{rel}}$, then $M_*=0$.
\end{proposition}
Note that for $N\ge3$, we are implicitly assuming $\lambda<4N/(N-2)$ since otherwise the two assumptions $q<2N/(2N+\lambda)$ and $q\ge1-2/N$ cannot be simultaneously satisfied. For the proof of Proposition~\ref{opt2} we need the following lemma which identifies the sub-leading term in~\eqref{eq:limitq}.
\begin{lemma}\label{bl}
Let $0<q<p$, let $f\in\mathrm L^p\cap\mathrm L^q(\R^N)$ be a symmetric non-increasing function and let $g\in\mathrm L^q(\R^N)$. Then, for any $\mu>0$, as $\epsilon\to 0_+$,
$$
\int_{\R^N}\left|f(x)+ \epsilon^{-N/p}\,\mu\,g(x/\epsilon)\right|^q\,dx = \int_{\R^N} f^q\,dx + \epsilon^{N(1-q/p)}\,\mu^q\int_{\R^N} |g|^q\,dx + o\(\epsilon^{N(1-q/p)}\,\mu^q\) \,.
$$
\end{lemma}
\begin{proof}[Proof of Lemma~\ref{bl}]
We first note that
\begin{equation}
\label{eq:fsingularity}
f(x) = o\(|x|^{-N/p}\)
\quad\text{as}\quad x\to 0
\end{equation}
in the sense that for any $c>0$ there is a $r>0$ such that for all $x\in\R^N$ with $|x|\le r $ one has $f(x)\le c\,|x|^{-N/p}$. To see this, we note that, since $f$ is symmetric non-increasing,
$$
f(x)^p\le\frac{1}{\left|\left\{ y\in\R^N:\ |y|\le|x|\right\}\right|} \int_{|y|\le|x|} f(y)^p\,dy \,.
$$
The bound~\eqref{eq:fsingularity} now follows by dominated convergence.

It follows from~\eqref{eq:fsingularity} that, as $\epsilon\to 0_+$,
$$
\epsilon^{N/p} f(\epsilon x) \to 0
\quad\text{for any}\quad x\in\R^N \,,
$$
and therefore, in particular, $\mu\,g(x) + \epsilon^{N/p} f(\epsilon x)\to \mu\,g(x)$ for any $x\in\R^N$. From the Br\'ezis--Lieb lemma (see~\cite{MR699419}) we know that
$$
\int_{\R^N} \left|\mu\,g(x) + \epsilon^{N/p} f(\epsilon x)\right|^q \,dx = \mu^q\int_{\R^N} |g(x)|^q \,dx + \int_{\R^N} \( \epsilon^{N/p} f(\epsilon x)\)^q \,dx + o(1) \,.
$$
By scaling this is equivalent to the assertion of the lemma. 
\end{proof}

\begin{proof}[Proof of Proposition~\ref{opt2}]
We argue by contradiction and assume that $M_*>0$. Let $0\le\sigma\in\(\mathrm L^1\cap\mathrm L^q\(\R^N\)\)\cap\mathrm L^1\(\R^N,|x|^\lambda\,dx\)$ with $\int_{\R^N} \sigma \,dx =1$. We compute the value of
$$
\mathcal Q[\rho,M] = \frac{I_\lambda[\rho] + 2M\Jlambda\rho x}{\left( \int_{\R^N} \rho(x)\,dx +M \right)^\alpha\left( \int_{\R^N} \rho(x)^q\,dx \right)^{(2-\alpha)/q}}
$$
for the family $(\rho,M)=\(\rho_* + \epsilon^{-N} \mu\,\sigma(\cdot/\epsilon),M^*-\mu\)$ with a parameter $\mu<M_*$.

\noindent 1) We have
\begin{multline*}
I_\lambda\left[\rho_* + \epsilon^{-N} \mu\,\sigma(\cdot/\epsilon)\right] + 2\,(M_*-\mu) \int_{\R^N} |x|^\lambda\left(\rho_*(x) + \epsilon^{-N} \mu\,\sigma(x/\epsilon)\right) dx \\
= I_\lambda[\rho_*] + 2\,M_*\Jlambda{\rho_*}x+ R_1
\end{multline*}
with
\begin{multline*}
R_1 = 2\,\mu \iint_{\R^N\times\R^N} \rho_*(x)\left( |x-y|^\lambda - |x|^\lambda\right) \epsilon^{-N} \sigma(y/\epsilon)\,dx\,dy \\
+ \epsilon^\lambda\,\mu^2\,I_\lambda[\sigma] + 2\,(M_*-\mu)\,\mu\,\epsilon^\lambda\Jlambda\sigma x\,.
\end{multline*}
Let us show that $R_1=O\(\epsilon^\beta\,\mu\)$ with $\beta:=\min\{2,\lambda\}$. This is clear for the last two terms in the definition of $R_1$, so it remains to consider the double integral. If $\lambda\le1$ we use the simple inequality $|x-y|^\lambda-|x|^\lambda\le|y|^\lambda$ to conclude that
$$
\iint_{\R^N\times\R^N} \rho_*(x)\left( |x-y|^\lambda - |x|^\lambda\right) \epsilon^{-N} \sigma(y/\epsilon)\,dx\,dy\le\epsilon^\lambda\Jlambda\sigma x\int_{\R^N}\rho_*\,dx\,.
$$
If $\lambda>1$ we use the fact that, with a constant $C$ depending only on $\lambda$,
$$
\label{eq:elemineq}
|x-y|^\lambda - |x|^\lambda\le -\lambda |x|^{\lambda-2} x\cdot y + C\left(|x|^{ (2-\lambda)_+ } |y|^\beta + |y|^\lambda \right).
$$
Since $\rho_*$ is radial, we obtain
\begin{multline*}
\iint_{\R^N\times\R^N} \rho_*(x)\left( |x-y|^\lambda - |x|^\lambda\right) \epsilon^{-N} \sigma(y/\epsilon)\,dx\,dy \\
\le C\left( \epsilon^\beta \int_{\R^N} |x|^{(2-\lambda)_+} \rho_*(x)\,dx \int_{\R^N} |y|^\beta \sigma(y)\,dy + \epsilon^\lambda\Jlambda\sigma x\int_{\R^N}\rho_*(x)\,dx\right).
\end{multline*}
Using H\"older's inequality and the fact that $\rho_*$, $\sigma\in\mathrm L^1\(\R^N\)\cap\mathrm L^1\(\R^N,|x|^\lambda\,dx\)$ it is easy to see that the integrals on the right side are finite, so indeed $R_1=O\(\epsilon^\beta\,\mu\)$.

\noindent 2) For the terms in the denominator of $\mathcal Q[\rho,M]$ we note that
$$
\int_{\R^N}\left( \rho_*(x) + \epsilon^{-N} \mu\,\sigma(x/\epsilon) \right)dx + (M_*-\mu) = \int_{\R^N} \rho_*\,dx + M_*
$$
and, by Lemma~\ref{bl} applied with $p=1$,
$$
\int_{\R^N}\left( \rho_*(x) + \epsilon^{-N} \mu\,\sigma(x/\epsilon) \right)^q\,dx = \int_{\R^N} \rho_*^q\,dx + \epsilon^{N\,(1-q)} \mu^q \int_{\R^N} \sigma^q\,dx + o\(\epsilon^{N\,(1-q)}\mu^q\) \,. 
$$
Thus,
\begin{multline*}
\left( \int_{\R^N}\left( \rho_*(x) + \epsilon^{-N} \mu\,\sigma(x/\epsilon) \right)^q\,dx \right)^{-\frac{2-\alpha}q}\\
=\left( \int_{\R^N} \rho_*^q\,dx \right)^{-\frac{2-\alpha}q}\left( 1- \frac{2-\alpha}q\,\epsilon^{N\,(1-q)} \mu^q\,\frac{\int_{\R^N} \sigma^q\,dx}{\int_{\R^N} \rho_*^q\,dx} + R_2 \right)
\end{multline*}
with $R_2= o\(\epsilon^{N\,(1-q)}\mu^q\)$.

\smallskip Now we collect the estimates. 
Since $(\rho_*,M_*)$ is a minimizer, we obtain that
\begin{multline*}
\mathcal Q\left[\rho_*+\epsilon^{-N} \mu\,\sigma(\cdot/\epsilon),M_*-\mu\right] = \mathcal C_{N,\lambda,q}\left( 1- \frac{2-\alpha}{q}\,\epsilon^{N\,(1-q)} \mu^q\,\frac{\int_{\R^N} \sigma^q\,dx}{\int_{\R^N} \rho_*^q\,dx} + R_2 \right) \\
+R_1\left( \int_{\R^N} \rho_*\,dx + M_* \right)^{-\alpha}\left( \int_{\R^N}\left( \rho_*(x) + \epsilon^{-N} \mu\,\sigma(x/\epsilon) \right)^q\,dx \right)^{-\frac{2-\alpha}q} \,.
\end{multline*}
If $\beta=\min\{2,\lambda\}>N\,(1-q)$, we can choose $\mu$ to be a fixed number in $(0,M_*)$, so that $R_1 =o\(\epsilon^{N\,(1-q)}\)$ and therefore
$$
\mathcal Q\left[\rho_*+\epsilon^{-N} \mu\,\sigma(\cdot/\epsilon),M_*-\mu\right]
\le\mathcal C_{N,\lambda,q}\left( 1- \frac{2-\alpha}q\,\epsilon^{N\,(1-q)} \mu^q\,\frac{\int_{\R^N} \sigma^q\,dx}{\int_{\R^N} \rho_*^q\,dx} + o\(\epsilon^{N\,(1-q)}\) \right).
$$
Since $\alpha<2$, this is strictly less than $\mathcal C_{N,\lambda,q}$ for $\epsilon>0$ small enough, contradicting the definition of $\mathcal C_{N,\lambda,q}$ as an infimum. Thus, $M_*=0$.

Note that if either $N=1$, $2$ or $N\ge3$ and $\lambda\le2N/(N-2)$, then the assumption $q>N/(N+\lambda)$ implies that $\beta>N\,(1-q)$. If $N\ge 3$ and $\lambda>2N/(N-2)$, then $\beta=2\ge N\,(1-q)$ by assumption. Thus, it remains to deal with the case where $N\ge 3$, $\lambda>2N/(N-2)$ and $2=N\,(1-q)$. In this case we have $R_1 = O\(\epsilon^2\,\mu\)$ and therefore
$$
\mathcal Q\left[\rho_*+\epsilon^{-N} \mu\,\sigma(\cdot/\epsilon),M_*-\mu\right]
\le\mathcal C_{N,\lambda,q}\left( 1- \frac{2-\alpha}q\,\epsilon^2\,\mu^q\,\frac{\int_{\R^N} \sigma^q\,dx}{\int_{\R^N} \rho_*^q\,dx} + O\(\epsilon^2\,\mu\) \right).
$$
By choosing $\mu$ small (but independently of $\epsilon$) we obtain a contradiction as before. This completes the proof of the proposition.
\end{proof}
\begin{remark}\label{rem:taylor}
In the proof of Proposition~\ref{opt2}, we used the bound $R_1=O\(\epsilon^2\,\mu\)$. For any $\lambda\ge2$, this bound is optimal. Namely, one has
\begin{multline*}
\iint_{\R^N\times\R^N} \rho_*(x)\left(|x-y|^\lambda - |x|^\lambda \right) \epsilon^{-N} \sigma(y/\epsilon)\,dx\,dy \\
= \epsilon^2 \ \frac{\lambda}{2}\left( 1+ \frac{\lambda-2}{N} \right) \int_{\R^N} |x|^{\lambda-2}\rho_*(x)\,dx \int_{\R^N} |y|^2 \sigma(y)\,dy + o\(\epsilon^2\)
\end{multline*}
for $\lambda\ge2$. This follows from the fact that, for any given $x\neq0$,
$$
|x-y|^\lambda - |x|^\lambda = -\lambda\,|x|^{\lambda-2} x\cdot y + \frac{\lambda}{2}\,|x|^{\lambda-2}\left( |y|^2 + (\lambda-2) \frac{(x\cdot y)^2}{|x|^2} \right) + O\(|y|^{\min\{3,\lambda\}} + |y|^\lambda\)\,.
$$
The assumption $\beta \geq N (1-q)$ is dictated by the $\epsilon^2$ behavior of $R_1$, for $\lambda\geq 2$, which cannot be improved.
\end{remark}

\section{Additional results}\label{Sec:AdditionalResults}

In this section we discuss the existence of a minimizer in the regime that is not covered by Proposition~\ref{opt2}. In particular, we will find a connection between the regularity of a minimizer of the relaxed problem and the presence or absence of a delta mass, and we will also establish the existence of a minimizer in a certain region which is not covered by Proposition~\ref{opt2}.
\begin{proposition}\label{prop:reg}
Let $N\ge3$, $\lambda>2N/(N-2)$ and $N/(N+\lambda)<q<\min\big\{1-2/N\,,\, 2N/(2N+\lambda)\big\}$. If $(\rho_*,M_*)$ is a minimizer for $\mathcal C_{N,\lambda,q}^{\rm{rel}}$ such that $(\rho_*,M_*) \in \mathrm L^{N\,(1-q)/2}(\R^N)\times[0,+\infty) $, then $M_*=0$.
\end{proposition}
The condition that the. minimizer $(\rho_*,M_*)$ of $\mathcal C_{N,\lambda,q}^{\rm{rel}}$ belongs to $\mathrm L^{N\,(1-q)/2}(\R^N)\times[0,+\infty)$ has to be understood as a regularity condition on $\rho_*$.
\begin{proof}
We argue by contradiction assuming that $M_*>0$ and consider a test function $\(\rho_* + \epsilon^{-N} \mu_\epsilon\,\sigma(\cdot/\epsilon),M^*-\mu_\epsilon\)$ such that $\irN\sigma=1$. We choose $\mu_\epsilon = \mu_1\,\epsilon^{N-2/(1-q)}$ with a constant $\mu_1$ to be determined below. As in the proof of Proposition~\ref{opt2}, one has
\begin{multline*}
I_\lambda\left[\rho_* + \epsilon^{-N} \mu_\epsilon\,\sigma(\cdot/\epsilon)\right] + 2(M_*-\mu_\epsilon) \int_{\R^N} |x|^\lambda\left(\rho_*(x) + \epsilon^{-N} \sigma(x/\epsilon)\right) dx \\
= I_\lambda[\rho_*] + 2M_*\Jlambda{\rho_*}x+ R_1
\end{multline*}
with $R_1=O\(\epsilon^2 \mu_\epsilon\)$. Note that here we have $\lambda\ge2$. For the terms in the denominator we note that
$$
\int_{\R^N}\left( \rho_*(x) + \epsilon^{-N} \mu_\epsilon\,\sigma(x/\epsilon) \right)dx + (M_*-\mu_\epsilon) = \int_{\R^N} \rho_*\,dx + M_*
$$
and, by Lemma~\ref{bl} applied with $p=N\,(1-q)/2$ and $\mu=\mu_\epsilon$, \emph{i.e.}, $\epsilon^{-N}\mu_\epsilon=\epsilon^{-N/p}\mu_1$, we have
$$
\int_{\R^N}\left( \rho_*(x) + \epsilon^{-N} \mu_\epsilon\,\sigma(x/\epsilon) \right)^q\,dx = \int_{\R^N} \rho_*^q\,dx + \epsilon^{N\,(1-q)} \mu_\epsilon^q \int_{\R^N} \sigma^q\,dx + o\(\epsilon^{N\,(1-q)}\mu_\epsilon^q\) \,. 
$$
Because of the choice of $\mu_\epsilon$ we have
$$
\epsilon^{N\,(1-q)} \mu_\epsilon^q = \epsilon^\gamma \mu_1^q
\quad\text{and}\quad
\epsilon^2 \mu_\epsilon = \epsilon^\gamma \mu_1
\quad\text{with}\quad\gamma := \frac{N-q\,(N+2)}{1-q}>0
$$
and thus
$$
\mathcal Q\left[\rho_*+\epsilon^{-N} \mu_\epsilon\,\sigma(\cdot/\epsilon),M_*-\mu_\epsilon\right]
\le\mathcal C_{N,\lambda,q}\left( 1- \frac{2-\alpha}q\,\epsilon^\gamma \mu_1^q\,\frac{\int_{\R^N} \sigma^q\,dx}{\int_{\R^N} \rho_*^q\,dx} + O\(\epsilon^\gamma \mu_1\) \right).
$$
By choosing $\mu_1$ small (but independent of $\epsilon$) we obtain a contradiction as before.
\end{proof}

Proposition~\ref{prop:reg} motivates the investigation of the regularity of the minimizer $(\rho_*,M_*)$ of $\mathcal C_{N,\lambda,q}^{\rm{rel}}$. We are not able to prove the regularity required in Proposition~\ref{prop:reg}, but we can state a dichotomy result which is interesting by itself.
\begin{proposition}\label{prop:cases}
Let $N\ge3$, $\lambda>2N/(N-2)$ and $N/(N+\lambda)<q<\min\big\{1-2/N\,,\,2N/(2N+\lambda)\big\}$. Let $(\rho_*,M_*)$ be a minimizer for $\mathcal C_{N,\lambda,q}^{\rm{rel}}$. Then the following holds:
\begin{enumerate}
\item If $\int_{\R^N}\rho_*\,dx > \frac{\alpha}{2}\,\frac{I_\lambda[\rho_*]}{\Jlambda{\rho_*}x}$, then $M_*=0$ and $\rho_*$ is bounded with
$$
\rho_*(0) =\left( \frac{(2-\alpha) I_\lambda[\rho_*] \int_{\R^N} \rho_*\,dx}{\left(\int_{\R^N} \rho_*^q\,dx\right)\left(2\Jlambda{\rho_*}x\int_{\R^N} \rho_*\,dx - \alpha I_\lambda[\rho_*]\right)} \right)^{1/(1-q)}
$$
\item If $\int_{\R^N}\rho_*\,dx = \frac{\alpha}{2}\,\frac{I_\lambda[\rho_*]}{\Jlambda{\rho_*}x}$, then $M_*=0$ and $\rho_*$ is unbounded with
$$
\rho_*(x) = C\,|x|^{-2/(1-q)} \,\big(1+o(1)\big)
\quad\text{as}\quad x\to 0
$$
for some $C>0$.
\item If $\int_{\R^N}\rho_*\,dx < \frac{\alpha}{2}\,\frac{I_\lambda[\rho_*]}{\Jlambda{\rho_*}x}$, then
$$
M_* = \frac{\alpha I_\lambda[\rho_*] - 2\Jlambda{\rho_*}x\ \int_{\R^N} \rho_*\,dx }{2\,(1-\alpha)\Jlambda{\rho_*}x}>0
$$
and $\rho_*$ is unbounded with
$$
\rho_*(x) = C\,|x|^{-2/(1-q)} \,\big(1+o(1)\big)
\quad\text{as}\quad x\to 0
$$
for some $C>0$.
\end{enumerate} 
\end{proposition}
Notice that the restrictions on $q$ and $\lambda$ in Proposition~\ref{prop:cases} are intended to cover the cases which are not already dealt with in Proposition~\ref{opt2}. The only assumptions that we shall use are $0<\alpha<1$ and $\lambda>2$. To prove Proposition~\ref{prop:cases}, let us begin with an elementary lemma.
\begin{lemma}\label{lem:minM}
For constants $A$, $B>0$ and $0<\alpha<1$, define
$$
f(M) = \frac{A+M}{(B+M)^\alpha} 
\quad\text{for}\quad M\ge0 \,.
$$
Then $f$ attains its minimum on $[0,\infty)$ at $M=0$ if $\alpha A\le B$ and at $M= (\alpha A-B)/(1-\alpha)>0$ if $\alpha A> B$.
\end{lemma}
\begin{proof}
We consider the function on the larger interval $(-B,\infty)$. Let us compute
$$
f'(M) = \frac{(B+M) - \alpha(A+M)}{(B+M)^{\alpha+1}} = \frac{B-\alpha A + (1-\alpha)M}{(B+M)^{\alpha+1}} \,.
$$
Note that the denominator of the right side vanishes exactly at $M=(\alpha A-B)/(1-\alpha)$, except possibly if this number coincides with $-B$.

We distinguish two cases. If $A\le B$, which is the same as $(\alpha A-B)/(1-\alpha)\le-B$, then~$f$ is increasing on $(-B,\infty)$ and then $f$ indeed attains its minimum on $[0,\infty)$ at $0$. Thus it remains to deal with the other case, $A>B$. Then $f$ is decreasing on $\big(-B,(\alpha A-B)/(1-\alpha)\big]$ and increasing on $\big[(\alpha A-B)/(1-\alpha),\infty\big)$. Therefore, if $\alpha A-B\le0$, then $f$ is increasing on $[0,\infty)$ and again the minimum is attained at $0$. On the other hand, if $\alpha A-B>0$, then $f$ has a minimum at the positive number $M=(\alpha A-B)/(1-\alpha)$.
\end{proof}

\begin{proof}[Proof of Proposition~\ref{prop:cases}.]
\emph{Step 1.} We vary $\mathcal Q[\rho_*,M]$ with respect to $M$. By the minimizing property the function
$$
M\mapsto \mathcal Q[\rho_*,M] = \frac{2\Jlambda{\rho_*}x}{\left( \int_{\R^N} \rho_*^q\,dx \right)^{(2-\alpha)/q}}\, \frac{A+M}{(B+M)^\alpha} 
$$
with
$$
A:= \frac{I_\lambda[\rho_*]}{2\Jlambda{\rho_*}x}
\quad\text{and}\quad
B := \int_{\R^N}\rho_*(x)\,dx
$$
attains its minimum on $[0,\infty)$ at $M_*$. From Lemma~\ref{lem:minM} we infer that
$$
M_* = 0
\quad\text{if and only if}\quad
\frac{\alpha}{2}\,\frac{I_\lambda[\rho_*]}{\Jlambda{\rho_*}x}\le\int_{\R^N}\rho_*(x)\,dx \,,
$$
and that $M_* = \frac{\alpha I_\lambda[\rho_*] - 2\(\Jlambda{\rho_*}x\)\(\int_{\R^N} \rho_*(y)\,dy\)}{2\,(1-\alpha)\Jlambda{\rho_*}x}$ if $\alpha\,\frac{I_\lambda[\rho_*]}{2\Jlambda{\rho_*}x} > \int_{\R^N}\rho_*(x)\,dx$.

\smallskip\noindent\emph{Step 2.} We vary $\mathcal Q[\rho,M_*]$ with respect to $\rho$. We begin by observing that $\rho_*$ is positive almost everywhere according to Lemma~\ref{sym-pos}. Because of the positivity of $\rho_*$ we obtain the Euler--Lagrange equation on $\R^N$,
$$
2\,\frac{\int_{\R^N} |x-y|^\lambda \rho_*(y)\,dy + M_* |x|^\lambda}{I_\lambda[\rho_*] + 2M_*\Jlambda{\rho_*}y} - \alpha\,\frac{1}{\int_{\R^N} \rho_*\,dy+M_*} - (2-\alpha)\, \frac{\rho_*(x)^{-1+q}}{\int_{\R^N} \rho_*(y)^q\,dy} = 0 \,.
$$
Letting $x\to 0$, we find that
$$
2\,\frac{\Jlambda{\rho_*}y}{I_\lambda[\rho_*] + 2M_*\Jlambda{\rho_*}y} - \alpha\,\frac{1}{\int_{\R^N} \rho_*(y)\,dy+M_*} = (2-\alpha)\,\frac{\rho_*(0)^{-1+q}}{\int_{\R^N} \rho_*(y)^q\,dy}\ge0 \,,
$$
with equality if and only if $\rho_*$ is unbounded. We can rewrite this as
$$
M_*\ge\frac{\alpha I_\lambda[\rho_*] - 2\left(\Jlambda{\rho_*}y\right)\left( \int_{\R^N} \rho_*\,dy\right)}{2\,(1-\alpha)\Jlambda{\rho_*}y}
$$
with equality if and only if $\rho_*$ is unbounded.

\smallskip\noindent\emph{Step 3.} Combining Steps 1 and 2 we obtain all the assertions of Proposition~\ref{prop:cases} except for the behavior of $\rho_*$ in the unbounded case. To compute the behavior near the origin we obtain, similarly as in Remark~\ref{rem:taylor}, using $\lambda>2$,
$$
\int_{\R^N} |x-y|^\lambda\rho_*(y)\,dy + M_*|x|^\lambda =\Jlambda{\rho_*}y + C\,|x|^2 \,\big(1+o(1)\big)
\quad\text{as}\quad x\to 0 \,,
$$
with
$$
C:= \frac12\,\lambda\,(\lambda-1) \int_{\R^N} |y|^{\lambda-2} \rho_*(y)\,dy \,.
$$
Thus, the Euler--Lagrange equation from Step 2 becomes
$$
\frac{2C\,|x|^2\,\big(1+o(1)\big)}{I_\lambda[\rho_*] + 2M_*\Jlambda{\rho_*}y} = (2-\alpha)\, \frac{\rho_*(x)^{-1+q}}{\int_{\R^N} \rho_*(y)^q\,dy}
\quad\text{as}\quad x\to 0 \,.
$$
This completes the proof of Proposition~\ref{prop:cases}.
\end{proof}
For any $\lambda>1$ we deduce from
$$
|x-y|^\lambda\le\big(|x|+|y|\big)^\lambda\le2^{\lambda-1}\,\big(|x|^\lambda+|y|^\lambda\big)
$$
that
$$
I_\lambda[\rho]<2^\lambda\Jlambda\rho x\int_{\R^N} \rho(x)\,dx\,.
$$
For all $\alpha\le2^{-\lambda+1}$, which can be translated into
$$
q>\frac{2N\,\big(1-2^{-\lambda}\big)}{2N\,\big(1-2^{-\lambda}\big)+\lambda}\,,
$$
Case~(1) of Proposition~\ref{prop:cases} applies and we know that $M_*=0$. Note that this bound for $q$ is in the range $\big(N/(N+\lambda)\,,\,2N/(2N+\lambda)\big)$ for all \hbox{$\lambda>1$}.

A better range for which $M_*=0$ can be obtained as follows when $N\ge3$. The superlevel sets of a symmetric non-increasing function are balls. From the layer cake representation we deduce that
$$
\label{lcr}
I_\lambda[\rho] \leq 2\,A_{N,\lambda}\Jlambda\rho x\int_{\R^N} \rho(x)\,dx
$$
where
$$
A_{N,\lambda}=\sup_{0\leq R,S<\infty}F(R,S)\quad\mbox{where}\quad F(R,S):=\frac{\iint_{B_R\times B_S} |x-y|^\lambda\,dx\,dy}{|B_R| \int_{B_S} |x|^\lambda\,dx + |B_S| \int_{B_R} |y|^\lambda\,dy}\,.
$$
For any $\lambda>1$, we have $2\,A_{N,\lambda}\le2^\lambda$, and also $A_{N,\lambda}\ge1/2$ because $I_\lambda[\1_{B_1}]\ge|B_1|\int_{B_1}|y|^\lambda\,dy$. The bound $A_{N,\lambda}\ge1/2$ can be improved to  $A_{N,\lambda}>1$ for any $\lambda>2$ as follows. We know that
$$
A_{N,\lambda}\ge F(1,1)=\frac{N\,(N+\lambda)}2\iint_{0\le r,\,s\le1}\kern-24pt r^{N-1}\,s^{N-1}\(\int_0^\pi\(r^2+s^2-2\,r\,s\,\cos\phi\)^{\lambda/2}\frac{\sin\phi^{N-2}\,d\phi}{W_N}\)dr\,ds\,.
$$
where $W_N$ is the Wallis integral $W_N:=\int_0^\pi\sin\phi^{N-2}\,d\phi$. For any $\lambda>2$, we can apply Jensen's inequality twice and obtain
\begin{multline*}
\int_0^\pi\(r^2+s^2-2\,r\,s\,\cos\phi\)^{\lambda/2}\frac{\sin\phi^{N-2}\,d\phi}{W_N}\\
\ge\(\int_0^\pi\(r^2+s^2-2\,r\,s\,\cos\phi\)\frac{\sin\phi^{N-2}\,d\phi}{W_N}\)^{\lambda/2}=\(r^2+s^2\)^{\lambda/2}
\end{multline*}
and
\begin{multline*}
\iint_{0\le r,\,s\le1}\kern-18pt r^{N-1}\,s^{N-1}\(r^2+s^2\)^{\lambda/2}dr\,ds\\
\ge\frac1{N^2}\(\iint_{0\le r,\,s\le1}\kern-18pt r^{N-1}\,s^{N-1}\(r^2+s^2\)N^2\,dr\,ds\)^{\lambda/2}=\frac1{N^2}\,\(\frac{2\,N}{N+2}\)^{\lambda/2}\,.
\end{multline*}
Hence
$$
A_{N,\lambda}\ge\frac{N+\lambda}{2\,N}\,\(\frac{2\,N}{N+2}\)^{\lambda/2}:=B_{N,\lambda}
$$
where $\lambda\mapsto B_{N,\lambda}$ is monotone increasing, so that $A_{N,\lambda}\ge B_{N,\lambda}>B_{N,2}=1$ for any $\lambda>2$. In this range we can therefore define
$$
\bar q(\lambda,N):=\frac{2N\,\big(1-A_{N,\lambda}^{-1}\big)}{2N\,\big(1-A_{N,\lambda}^{-1}\big)+\lambda}\,.
$$
Based on a numerical computation, the curve $\lambda\mapsto\bar q(\lambda,N)$ is shown on Fig.~\ref{Fig1}. The next result summarizes the above considerations.
\begin{proposition}\label{Cor:Mstar} Assume that $N\ge3$ and $\lambda>2N/(N-2)$. Then, with the above notations,
$$
\bar q(\lambda,N)\le\frac{2N\,\big(1-2^{-\lambda}\big)}{2N\,\big(1-2^{-\lambda}\big)+\lambda}<\frac{2\,N}{2\,N+\lambda}
$$
and, for $\lambda>2$ large enough,
$$
\bar q(\lambda,N)>\frac N{N+\lambda}\,.
$$
If $q$ is such that $\max\big\{\bar q(\lambda,N),N/(N+\lambda)\big\}<q<\min\big\{1-2/N\,,\,2N/(2N+\lambda)\big\}$ and if $(\rho_*,M_*)$ is a minimizer for $\mathcal C_{N,\lambda,q}^{\rm{rel}}$, then $M_*=0$ and $\rho_*\in\mathrm L^\infty(\R^N)$.\end{proposition}

\setlength\unitlength{1cm}
\begin{figure}[!ht]\begin{center}\begin{picture}(12,8)
\put(0,0){\includegraphics[width=12cm]{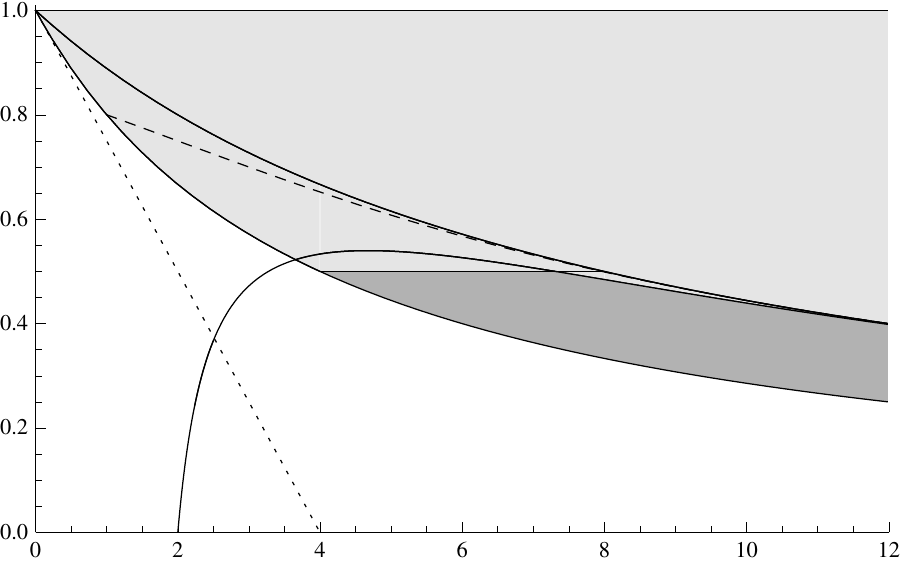}}
\put(11.5,0.5){$\lambda$}
\put(-0.25,6.5){$q$}
\put(0.5,3.7){$q=\frac{N-2}N$}
\put(3.8,4.25){$\scriptstyle q=\bar q(\lambda,N)$}
\put(10,3.55){$q=\frac{2N}{2N+\lambda}$}
\put(4.9,2.7){$q=\frac N{N+\lambda}$}
\end{picture}
\caption{\label{Fig1}\scriptsize Main regions of the parameters (here $N=4$). The case $q=2N/(2N+\lambda)$ has already been treated in~\cite{DZ15,MR3666824}. Inequality~\eqref{ineq:rHLS} holds with a positive constant $\mathcal C_{N,\lambda,q}$ if $q>N/(N+\lambda)$, which determines the admissible range corresponding to the grey area, and it is achieved by a function $\rho$ (without any Dirac mass) in the light grey area. The dotted line is $q=1-\lambda/N$: it is tangent to the admissible range of parameters at $(\lambda,q)=(0,1)$. In the dark grey region, Dirac masses with $M_*>0$ are not excluded. The dashed curve corresponds to the curve $q=2N\,\big(1-2^{-\lambda}\big)\big/\big(2N\big(1-2^{-\lambda}\big)+\lambda\big)$ and can hardly be distinguished from $q=2N/(2N+\lambda)$ when $q$ is below $1-2/N$. The curve $q=\bar q(\lambda,N)$ of Corollary~\ref{Cor:Mstar} is also represented . Above this curve, no Dirac mass appears when minimizing the relaxed problem corresponding to~\eqref{ineq:rHLS}. Whether Dirac masses appear in the region which is not covered by Corollary~\ref{Cor:Mstar} is an open question.}
\end{center}\end{figure}
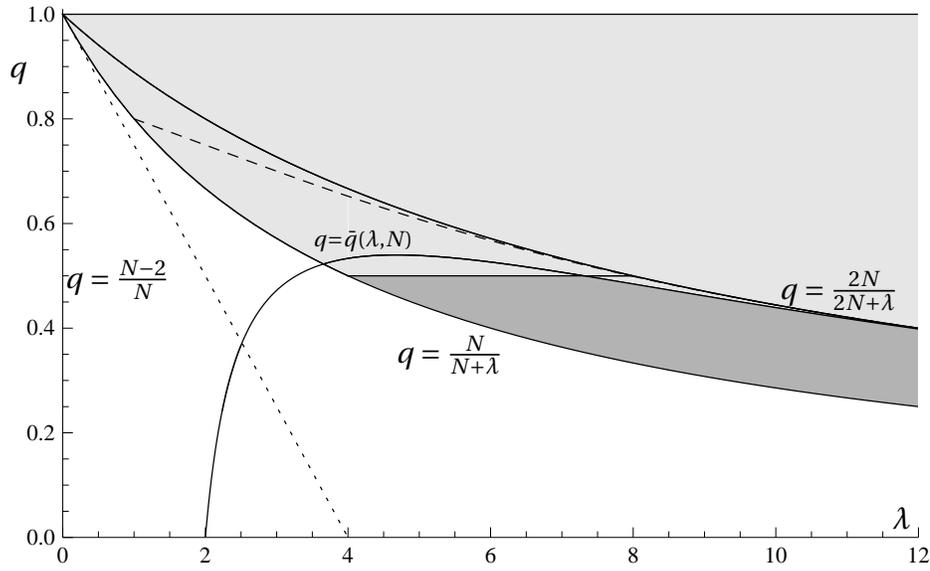

\subsection*{Acknowledgments}\begin{spacing}{0.75}{\noindent\scriptsize This research has been partially supported by the projects \emph{EFI}, contract~ANR-17-CE40-0030 (J.D.) and \emph{Kibord}, contract~ANR-13-BS01-0004 (J.D., F.H.) of the French National Research Agency (ANR), and by the U.S. National Science Foundation through grant DMS-1363432 (R.L.F.). The research stay of F.H.~in Paris in December 2017 was partially supported by the Simons Foundation and by Mathematisches Forschungsinstitut Oberwolfach. Some of the preliminary investigations were done at the Institute Mittag-Leffler during the fall program \emph{Interactions between Partial Differential Equations \& Functional Inequalities}. The authors thank J.A.~Carrillo for preliminary discussions which took place there and R.L.F.~thanks the University Paris-Dauphine for hospitality in February 2018.\\[2pt]
\copyright\,2018 by the authors. This paper may be reproduced, in its entirety, for non-commercial purposes.}\end{spacing}


\end{document}